\documentclass[11pt]{amsart}%fleqn
\usepackage{fullpage}

\usepackage{amsmath}
\usepackage{amssymb}
\usepackage{amsfonts}
\usepackage{graphicx}
\usepackage{amsthm}
\usepackage{enumerate}
\usepackage{lscape}
\usepackage{dsfont}
\usepackage{color}
\usepackage{mathtools}
\usepackage[dvipsnames]{xcolor}
\usepackage{hyperref}
\usepackage{appendix}

\newcounter{example}

\usepackage[utf8]{inputenc}

\newcommand{\C}{\mathds C}
\newcommand{\R}{\mathds R}

\usepackage{apptools}

\newtheorem{theorem}{Theorem}

\newtheorem{lemma}{Lemma}

\newtheorem{corollary}{Corollary}

\newtheorem{remark}{Remark}

\newtheorem{question}{Q}

\numberwithin{equation}{section}

\def\subsubsection{\@startsection{subsubsection}{3}%
\z@{.5\linespacing\@plus.7\linespacing}{-.5em}%
{\normalfont\bfseries}}
\makeatother

\title{On the third coefficient in the TYCZ--expansion of the epsilon function of K\"ahler--Einstein manifolds}

\author{Simone Cristofori}
\address{(Simone Cristofori) Dipartimento di Scienze Matematiche, Fisiche e Informatiche \\
         Universit\`a di Parma (Italy)}
         
        \email{simone.cristofori@unipr.it}
\author{Michela Zedda}
\address{(Michela Zedda) Dipartimento di Scienze Matematiche, Fisiche e Informatiche \\
         Universit\`a di Parma (Italy)}
\email{michela.zedda@unipr.it}
\date{\today}
%53C07, 53C42}

\thanks{
The second named author has been supported by the project Prin 2022 – Real and Complex Manifolds: Geometry and Holomorphic Dynamics – Italy. Both the authors were supported by INdAM GNSAGA - Gruppo Nazionale per le Strutture Algebriche, Geometriche e le loro Applicazioni.} 

\begin{document}

\maketitle

\begin{abstract}
In this paper we compute the third coefficient arising from the TYCZ-expansion of the $\varepsilon$-function associated to a K\"ahler-Einstein metric and discuss the consequences of its vanishing.
\end{abstract}

\section{Introduction and statement of the main result}
Let $(M,g)$ be a K\"ahler manifold and assume that the K\"ahler form $\omega$ associated to $g$ is integral. This condition guarantees the existence of a holomorphic line bundle $L\rightarrow M$ over $M$ such that for any $m\in\mathds Z^+$ one can define a hermitian metric $h_m$ on $L^m:=L^{\otimes m}$ that satisfies ${\rm Ric}(h_m)=m\omega$. Recall that ${\rm Ric}(h_m)$ is the form whose local expression in a trivializing open set $U\subset M$ is given by:
$$
{\rm Ric}(h_m)=-\frac i2 \partial\bar \partial \log h_m(\sigma(x),\sigma(x)),
$$
for a trivializing holomorphic section $\sigma:U\rightarrow L\setminus \{0\}$. 
Consider the complex Hilbert space $\mathcal H_m$ of global holomorphic sections of $L^m$ that are $L^2$-limited in norm with respect to the product:
$$
\langle s,s\rangle_{m}:=\int_Mh_m(s(x),s(x))\frac{\omega^n}{n!}.
$$
Denote $\dim \mathcal H_m=d_m+1$. Observe that when $M$ is compact, $\mathcal H_m=H^0(L^m)$ and we have $0<\dim H_m<\infty$. However in the noncompact case, $\mathcal H_m$ could reduce to be $\{0\}$ or its dimension could be infinite.
When $\mathcal H_m\neq \{0\}$, we can pick an orthonormal basis $\{s_j\}$ of $\mathcal H_m$ and define:
$$
\varepsilon_{m g}(x):=\sum_{j=0}^{d_m}h_m(s_j(x),s_j(x)).
$$ 
This function is globally defined and independent on the orthonormal basis chosen or on the hermitian metric (see e.g. \cite{cgr1}), and in literature it also appears under the name of {\em distortion function}.

When $M$ is compact, D. Catlin \cite{catlin} and independently S. Zelditch \cite{zelditch} proved that the function $\varepsilon_{m g}$ admits an asymptotic expansion, known as Tian-Yau-Catlin-Zelditch (from now on TYCZ) expansion:
\begin{equation}\label{eq:sviluppo asintotico}
\epsilon_{m g}(x)\sim\sum_{j=0}^\infty a_j(x) m^{n-j},
\end{equation}
that is for every pair of integers $l$, $r$, there exists a constant $C(l,r)>0$ such that:
\begin{equation}\label{asympt}
\left|\left|\epsilon_{m g}(x)-\sum_{j=0}^l a_j(x) m^{n-j} \right|\right|_{C^r}\le \frac{C(l,r)}{m^{l+1}}.
\end{equation}
Here $a_0(x)\equiv 1$ and the $a_j(x)$, $j=1,2,\dots$ are smooth functions on $M$.
In \cite{lu} Z. Lu computed the first three coefficients of this expansion and he proved that each of the coefficients $a_j(x)$ is a polynomial of the curvature and its covariant derivatives of the metric $g$ which can be determined by finitely many algebraic operations. The expression of the first three coefficients is a key tool in our analysis and it is given in the next section. 

When $M$ is noncompact, we say that a TYCZ--expansion \eqref{eq:sviluppo asintotico} exists if \eqref{asympt} holds for any compact subset of $K\subset M$, as introduced in \cite{arezzoloi} by C. Arezzo and A. Loi. We observe that in this case the construction make sense also for nonintegers values of $m$, that for this reason to avoid confusion will be denoted by $\alpha\in \mathds R^+$. In constrast with the compact setting, the existence of such an expansion is not guaranteed. In \cite{engliscoeff} M. Engli\v{s} proved the existence of a TYCZ--expansion in the case of strongly pseudoconvex domains of $\mathds C^n$ with real analytic boundary, and computed the $a_j$'s coefficients obtaining the same results as Lu. A more general result can be obtained by the work of X. Marinescu and G. Marinescu \cite[Thm. 6.1.1]{mamarinescu}, as described in the next section.

Studying metrics with the TYCZ coefficients being prescribed is a very natural generalization of the problem of finding K\"ahler metrics with constant scalar curvature on a K\"ahler manifold.
The vanishing of the coefficients $a_k$ for $k\ge n+1$ turns out to be related to some important problems in the theory of pseudoconvex manifolds (cf. \cite{LT,ALZ,LUZ}).
In the noncompact setting, one can find in \cite{LZmama} a characterization of the flat metric among locally hermitian symmetric spaces as the only one with vanishing $a_1$ and $a_2$, while in \cite{FT} Z. Feng and Z. Tu solve a conjecture formulated in \cite{Z} by showing that the complex hyperbolic space is the only Cartan-Hartogs domain where the coefficient $a_2$ is constant.

The first result of this paper is the following theorem, where we investigate the consequences of the vanishing of $a_3$ in the K\"ahler--Einstein setting. 

\begin{theorem}\label{main1}
   Let $(M,g)$ be a $n$-dimensional K\"ahler--Einstein manifold with integral K\"ahler form $\omega$. If $\mathcal H_{\alpha}\neq \{0\}$, there exists a TYCZ--expansion for $\varepsilon_{\alpha g}$ whose coefficients satisfy the following:
   \begin{enumerate}
\item if $n=2$ then $a_3=0$ if and only if $\Delta|R|^2=0$;
\item if $n\geq 3$ then $a_3= 0$ implies $g$ is Ricci-flat.
   \end{enumerate}
\end{theorem}

Observe that the vanishing of $a_2$ (for $n\geq 2$) readily implies that the metric is flat.
In \cite[Thm. 1.1]{loisaliszuddas}, A. Loi, F. Salis and F. Zuddas prove that {\em the coefficient $a_3$ of a radial constant scalar curvature K\"ahler metric is constant if and only if $a_2$ is constant} (see Thm. \ref{LSZ} below). Combining this result with Thm. \ref{main1} we obtain the following:
\begin{corollary}\label{radialcor}
    Let $(M,g)$ be a K\"ahler--Einstein manifold endowed with a radial K\"ahler metric $g$. Then $a_3=0$ if and only if $(M,g)$ is biholomorphically isometric to $\mathds C^n$, $\mathds C{\rm H}^2$ or $\mathds C{\rm P}^2$ with (a multiple of) their standard metric.
\end{corollary}

In \cite{TwoConjectures}, Loi, Salis, and Zuddas, conjecture that a Ricci--flat metric on a $n$-dimensional complex manifold such that $a_{n+1}=0$ is forced to be flat. By Theorem \ref{main1} such conjecture is equivalent, in the $n=2$ case, to proving that a Ricci--flat surface with harmonic $|R|^2$ is flat. 
Furthermore, it follows from the proof of Theorem \ref{main1} that Ricci--flat metrics or K\"ahler--Einstein metrics on surfaces which are either homogeneous or regular have vanishing $a_3$. 
 Here homogeneous means that the group of isometric automorphisms of $(M,g)$ acts transitively on $M$, while a regular metric is a metric whose $\varepsilon$-function $\varepsilon_{mg}$ is constant for all large enough $m$. 
 It is an open question to understand if regular K\"ahler--Einstein metrics are homogeneous. If one drop the K\"ahler--Einstein assumption one gets a negative answer, at least for noncompact manifolds, as in \cite{CAL} F. Cannas Aghedu and A. Loi proved that the scalar flat Burn--Simanca metric is a regular nonhomogeneous metric on the blow-up of $\mathds C^2$ at one point. The question if a regular K\"ahler manifold is homogeneous is still an interesting and open question in the compact setting, where the manifolds involved are projective algebraic.

This leads us to the following question:
\begin{question}\label{conj}
    Is a K\"ahler-Einstein manifold with vanishing $a_3$ homogeneous?
\end{question}
In view of Theorem \ref{main1}, a positive answer to this question would imply that the only nonflat K\"ahler-Einstein manifolds with vanishing $a_3$ are homogeneous surfaces. Observe that we can construct many examples of metrics with vanishing $a_3$ by taking the K\"ahler product of a bounded symmetric domain with its Bergman metric times its compact dual (see \cite{LZmama}), however the product metric is not K\"ahler-Einstein. Notice also that the family of Taub-NUT metrics on $\mathds C^2$ are an example of Ricci-flat nonhomogeneous metrics with $a_3\neq 0$ (see \cite{loizeddazuddas}). 

As second result of this paper, we compute the coefficient $a_3$ of a locally nonhomogeneous complete K\"ahler--Einstein manifold constructed by Calabi in \cite{calabi}, proving the following:
\begin{theorem}\label{main2}
The $\varepsilon$-function associated to Calabi's metric admits a TYCZ--expansion with $a_3\neq 0$.
\end{theorem}

The paper consists of two other sections. In the next one we compute the coefficients $a_2$ and $a_3$ for K\"ahler--Einstein manifolds, from Lu's formulas, proving Theorem \ref{thm:main} and Corollary \ref{radialcor}. Last section is devoted to Calabi's inhomogeneous metric and the proof of Theorem \ref{main2}.

\section{The vanishing of $a_3$ for K\"ahler--Einstein manifolds}\label{sec:expansion}

We begin recalling a result of X. Ma and G. Marinescu \cite[Theorem 6.1.1]{mamarinescu} adapted to our setting, that is, in Marinescu and Ma notations, we set $E=\mathds C$ and $\Theta=\omega=iR^L$ and we observe that $iR^{\det}$ is the Ricci form $\rho$ of $(M,\omega)$ (see also \cite[Thm. 5]{CZ} and \cite[Thm. 7]{loizeddazuddas} for a contractible version).
\begin{theorem}\label{th: mamarinescu}
Let $(M,\omega)$ be a K\"ahler manifold with associated Ricci form $\rho$, and $L\rightarrow M$ be a holomorphic line bundle such that $c_1(L)=[\omega]$. If there exist $c>0$ such that $\rho+c\omega$ is positive definite,
then $\varepsilon_{\alpha g}$ admits a TYCZ-expansion in $\alpha$.
\end{theorem}

It is clear that in the K\"ahler--Einstein case, that is when $\rho=\lambda \omega$, the existence of a TYCZ-expansion for $\varepsilon_{\alpha g}$ is always guaranteed. However, observe that to construct the $\varepsilon$--function we need $\mathcal H_\alpha\neq \{0\}$. Such condition is satisfied for example when $M$ has finite volume, since in this case the constant functions belongs to $\mathcal H_\alpha$, but of course there are manifolds with infinite volume such that $\mathcal H_\alpha\neq \{0\}$ (see e.g. Calabi's manifold in the next Section).

In order to prove Theorem \ref{main1} we need the following formulas, computed by Z. Lu \cite{lu} for compact manifolds and by M. Engli\v s \cite{engliscoeff} for noncompact ones:
\begin{equation}
\begin{cases} \label{coeffespan} a_1=-\frac12\sigma\\ a_2=-\frac13\Delta\sigma+\frac1{24}\left(|R|^2-4|{\rm Ric}|^2+3\sigma^2\right)\\
    \begin{split}
        a_3=&-\frac18\Delta\Delta\sigma+\frac1{24}{\rm div}{\rm div}(R,{\rm Ric})-\frac16{\rm div}({\rm div}(\sigma {\rm Ric}))+\frac1{48}\Delta(|R|^2-4|{\rm Ric}|^2+8\sigma^2)+\\
        &-\frac1{48}\sigma(\sigma^2+|R|^2-4|{\rm Ric}|^2){-}\frac1{24}(\sigma_3({\rm Ric})-{\rm Ric}(R,R){+}R({\rm Ric},{\rm Ric}))
    \end{split}
    \end{cases}
\end{equation}
where $R$ denotes the Riemannian tensor of $g$, given in local coordinates by:
\[
R_{i\overline jk\overline l}=\frac{\partial^2 g_{i\overline j}}{\partial z_k\partial\overline z_l} -\sum_{p,q=1}^n g^{p\overline q}\frac{\partial g_{i\overline q}}{\partial z_k}\frac{\partial g_{p\overline j}}{\partial\overline z_l},
\]
${\rm Ric}$ denotes the Ricci tensor, that (in contrast with Lu's notation that has the opposite sign) reads:
$$
{\rm Ric}_{i\bar j}=g^{l\bar m}R_{i\bar jl\bar m},
$$
and:
$$
\sigma=g^{i\bar j}{\rm Ric}_{i\bar j},
$$
is the scalar curvature. All the norms are taken with respect to $g$, and we further are using the following notations (here again the signs has be changed accordingly to our notation):
\begin{equation}\label{lu}
    \begin{split}
        & |D'\sigma|^2=g^{j\overline{i}}\frac{\partial\sigma}{\partial z_i}\frac{\partial\sigma}{\partial \overline{z_j}},\qquad  |D'{\rm Ric}|^2=g^{\alpha\overline{i}}g^{j\overline{\beta}}g^{\gamma\overline{k}}Ric_{i\overline{j},k}  \overline{{\rm Ric}_{\alpha\overline{\beta},\gamma}},\\
        &|D'R|^2=g^{\alpha\overline{i}}g^{j\overline{\beta}}g^{\gamma\overline{k}}g^{l\overline{\delta}}g^{\epsilon\overline{p}}R_{i\overline{j}k\overline{l},p}\overline{R_{\alpha\overline{\beta}\gamma\overline{\delta},\epsilon}},\\
        &
        {\rm div}{\rm div}(\sigma {\rm Ric})=2|D'\sigma|^2+g^{\beta\overline{i}}g^{j\overline{\alpha}}{\rm Ric}_{i\overline{j}}\frac{\partial^2\sigma}{\partial z_{\alpha}\partial\overline{z_\beta}}+\sigma\Delta\sigma,\\
        &
           \sigma_3({\rm Ric})=g^{\delta\overline{i}}g^{j\overline{\alpha}}g^{\beta\overline{\gamma}}{\rm Ric}_{i\overline j}{\rm Ric}_{\alpha\overline\beta}{\rm Ric}_{\gamma\overline \delta},\\
        &
        R({\rm Ric},{\rm Ric})=g^{\alpha \overline{i}}g^{j\overline\beta}g^{\gamma \overline k}g^{l\overline\delta}R_{i\overline jk\overline l}{\rm Ric}_{\beta\overline{\alpha}}{\rm Ric}_{\delta\overline{\gamma}},\\
        &
        {\rm Ric}(R,R)=g^{\alpha \overline i}g^{j\overline\beta}g^{\gamma \overline k}g^{\delta \overline p}g^{q\overline\epsilon}{\rm Ric}_{i\overline j}R_{\beta\overline\gamma p \overline q}R_{k\overline\alpha\epsilon\overline\delta},\\
        &{\rm div}{\rm div}(R,{\rm Ric})=-g^{\beta\overline{i}}g^{j\overline{\alpha}}Ric_{i\overline{j}}\frac{\partial^2\sigma}{\partial z_{\alpha}\partial\overline{z_\beta}}-2|D'{\rm Ric}|^2{-}g^{\alpha\overline{i}}g^{j\overline{\beta}}g^{\gamma\overline{k}}g^{l\overline{\delta}}{\rm Ric}_{i\overline{j},k\overline{l}}R_{\beta\overline{\alpha}\delta\overline{\gamma}}+\\
        &\qquad\qquad\qquad\quad\  -R({\rm Ric},{\rm Ric}){+}\sigma_3({\rm Ric}),
    \end{split}
\end{equation}
where $",p"$ represents the covariant derivative with respect to $\frac{\partial}{\partial z_p}$ and $\Delta$ represents the Laplace operator
\begin{equation}\label{laplace}
\Delta=\sum_{i,j=1}^n g^{i\overline j}\frac{\partial^2}{\partial z_i\partial\overline{z}_j}.
\end{equation}

For a K\"ahler--Einstein manifold $(M,g)$ with Einstein constant $\lambda$, we have
\[
{\rm Ric}=\lambda g,\qquad 
\sigma=n\lambda.
\]

Further
    \begin{equation}\label{div}
    {\rm div}{\rm div}(R,{\rm Ric})=-R({\rm Ric},{\rm Ric})+\sigma_3({\rm Ric}),
    \end{equation}
    since by \eqref{lu}, using that $\sigma$ is constant, we get
    \[
    {\rm div}{\rm div}(R,{\rm Ric})=-2|D'{\rm Ric}|^2-g^{\alpha\overline{i}}g^{j\overline{\beta}}g^{\gamma\overline{k}}g^{l\overline{\delta}}{\rm Ric}_{i\overline{j},k\overline{l}}R_{\beta\overline{\alpha}\delta\overline{\gamma}}-R({\rm Ric},{\rm Ric})+\sigma_3({\rm Ric}),
    \]
    and both the terms $|D'{\rm Ric}|^2$ and $g^{\alpha\overline{i}}g^{j\overline{\beta}}g^{\gamma\overline{k}}g^{l\overline{\delta}}{\rm Ric}_{i\overline{j},k\overline{l}}R_{\beta\overline{\alpha}\delta\overline{\gamma}}$ involve the covariant derivatives of the Ricci tensor, that vanish since the metric is K\"ahler--Einstein (${\rm Ric}=\lambda g$ and $\nabla g=0$).\\

The following lemma is a key step in the proof of Theorem \ref{main1}.

\begin{lemma}\label{thm:main}
Let $(M,g)$ be a $n$-dimensional K\"ahler--Einstein manifold with Einstein constant $\lambda$. Then:
\begin{enumerate}
\item  $a_2=\frac1{24}\left(|R|^2+n\lambda^2(3n-4)\right)$;
\item  $a_3=\frac1{48}\left(\Delta|R|^2-\lambda(n-2)(\lambda^2n(n-2)+|R|^2)\right).$
\end{enumerate}
\end{lemma}

\begin{proof}%[{Proof of Proposition \ref{thm:main}}]
Setting $\sigma=\lambda n$ and since 
\begin{equation*}
        \begin{split}
           |{\rm Ric}|^2&=g^{i\overline{k}}g^{l\overline{j}} {\rm Ric}_{i\overline{j}}\overline{Ric_{k\overline{l}}}% = g^{i\overline{k}}g^{l\overline{j}} Ric_{i\overline{j}}Ric_{\overline{k}l}\\
           =g^{i\overline{k}}g^{l\overline{j}} {\rm Ric}_{i\overline{j}}{\rm Ric}_{l\overline{k}}=\lambda^2\,
           g^{i\overline{k}}g^{l\overline{j}} g_{i\overline{j}} g_{l\overline{k}}=\lambda^2 n,
        \end{split}
    \end{equation*}

\eqref{coeffespan} reads:
\[
a_2=\frac1{24}\left(|R|^2+n\lambda^2(3n-4)\right);
\]
\begin{equation}\label{eq:a3KE}
\begin{split}
a_3=&\frac1{24}{\rm div}{\rm div}(R,{\rm Ric})+\frac1{48}\Delta|R|^2-\frac1{48}\lambda n(\lambda^2n^2+|R|^2-4\lambda^2n)+\\
&-\frac1{24}(\sigma_3({\rm Ric})-{\rm Ric}(R,R)+R({\rm Ric},{\rm Ric})).
        \end{split}
        \end{equation}
By \eqref{div} and since:
  \begin{equation}\label{eq:divdiv}
        \begin{split}
       {\rm Ric}(R,R)=&g^{\alpha \overline i}g^{j\overline\beta}g^{\gamma \overline k}g^{\delta \overline p}g^{q\overline\epsilon}{\rm Ric}_{i\overline j}R_{\beta\overline\gamma p \overline q}R_{k\overline\alpha\epsilon\overline\delta}=\lambda g^{j\overline\beta}g^{\gamma \overline k}g^{\delta \overline p}g^{q\overline\epsilon}R_{\beta\overline\gamma p \overline q}R_{k\overline j\epsilon\overline\delta}=\lambda|R|^2,\\    
       R({\rm Ric},{\rm Ric})=&g^{\alpha \overline i}g^{j\overline\beta}g^{\gamma \overline k}g^{l\overline\delta}R_{i\overline jk\overline l}{\rm Ric}_{\beta\overline\alpha}{\rm Ric}_{\delta\overline\gamma}%=\lambda^2g^{\alpha \overline i}g^{j\overline\beta}g^{\gamma \overline k}g^{l\overline\delta}R_{i\overline jk\overline l}g_{\beta\overline\alpha}g_{\delta\overline\gamma}\\
            =\lambda^2g^{j\overline\beta}g^{\gamma\overline k}R_{\beta \overline jk\overline\gamma}=\lambda^2g^{\gamma \overline k}{\rm Ric}_{k\overline\gamma}=\lambda^2\sigma=n\lambda^3,
   \end{split}
    \end{equation}
substituting in \ref{eq:a3KE}, we obtain
    \begin{equation*}
    \begin{split}
        =&\frac1{24}(-R({\rm Ric},{\rm Ric})+\sigma_3(Ric))+\frac1{48}\Delta(|R|^2)-\frac1{48}\lambda n(\lambda^2 n^2+|R|^2 -4\lambda^2n)+\\
        &-\frac1{24}(\sigma_3({\rm Ric})-{\rm Ric}(R,R)+R({\rm Ric},{\rm Ric}))\\
        =&\frac1{24}(-2R({\rm Ric},{\rm Ric})+{\rm Ric}(R,R))+\frac1{48}\Delta(|R|^2)-\frac1{48}\lambda n(\lambda^2 n^2+|R|^2-4\lambda^2n)\\
        =&\frac1{24}(-2n\lambda^3+\lambda|R|^2)+\frac1{48}\Delta(|R|^2)-\frac1{48}\lambda n(\lambda^2 n^2+|R|^2-4\lambda^2n)\\ 
        =&\frac1{48}\left(\Delta|R|^2-\lambda(n-2)(\lambda^2n(n-2)+|R|^2)\right),
    \end{split}
\end{equation*}
concluding the proof.
\end{proof}

The proof of Theorem \ref{main1} follows now by Theorem \ref{th: mamarinescu} and by Lemma \ref{thm:main}. More precisely:
\begin{proof}[Proof of Theorem \ref{main1}]
The existence of a TYCZ--expansion follows directly from Theorem \ref{th: mamarinescu}, while (1) and (2) follow readily from Lemma \ref{thm:main}.
   To prove (3) assume $n>2$. Then:
  $$
  0=\int_Ma_3\frac{\omega^n}{n!}=-\frac1{48}\lambda(n-2)\int_M\left(\lambda^2n(n-2)+|R|^2\right)\frac{\omega^n}{n!},
  $$
  and since the integrand function is nonnegative, to the right hand side to be zero we need $\lambda=0$.
\end{proof}

Let us deal now with the proof of Corollary \ref{radialcor}.
A K\"ahler metric $\omega$ on a K\"ahler manifold $M$ is said to be a {\em radial metric} if for any point $p\in M$ there exists a coordinate neighborhood $U$ of $p$ such that $\omega|_U$ can be described by a K\"ahler potential which depends only on the sum $|z|^2 = |z_1|^2 +\dots +|z_n|^2$ of the moduli of the local coordinates.
In \cite{loisaliszuddas}, Loi, Salis and Zuddas proved the following Theorem:
\begin{theorem}[A. Loi, F. Salis, F. Zuddas]\label{LSZ}
    The third coefficient $a_3$ of a radial constant scalar curvature K\"ahler metric is constant if and only if the second coefficient $a_2$ is constant.
\end{theorem}
\begin{proof}[Proof of Corollary \ref{radialcor}]
Assume $(M,\omega)$ to be K\"ahler--Einstein and radial. If $a_3=0$, then by Theorem \ref{LSZ} $a_2$ is constant. From Z. Feng classification \cite{feng} (see also \cite[Thm. 2.1]{loisaliszuddas}) of radial constant scalar curvature potentials with constant $a_2$, the only K\"ahler--Einstein are (a multiple of) the standard metrics on $\mathds C^n$, $\mathds C{\rm H}^n$ or $\mathds C{\rm P}^n$. The cases $\mathds C{\rm H}^n$ and $\mathds C{\rm P}^n$, with $n\geq 3$, are excluded by Theorem \ref{main1}.
\end{proof}

\section{The coefficient $a_3$ for Calabi's inhomogeneous metric}\label{sec:calabi}
In \cite{calabi}, E. Calabi constructs the following complete not locally homogeneous K\"ahler--Einstein metric. Consider the complex tubular domains $M_n:=\frac{1}{2}D\oplus i\R^n\subset\C^n$, where $D$ is an open ball in $\R^n$ centred in the origin and of radius $a$. Let $g_n$ be the K\"ahler metric on $M_n$ whose K\"ahler form is given by
\[
\omega_n=\frac{i}{2}\partial\overline\partial F(z,\overline{z}),
\]
with
\[
F(z,\overline{z})=f(z_1+\overline{z_1},\dots,z_n+\overline{z_n})
\]
where $f:D\to\R$ is a strongly convex, % (i.e. the Hessian is everywhere positive definite), 
differentiable and radial function, i.e. $f(x_1,\dots,x_n)=y(r)$, with $r=\sqrt{\sum_{j=1}^n x_j^2}$, that diverges uniformly at $+\infty$ at the boundary of $D$ and $x_\alpha=z_\alpha+\overline{z_\alpha}$, for $\alpha=1,\dots,n$. The function $f$ satisfies the following Cauchy problem
\begin{equation}\label{cauchy}
\begin{cases}
    (\frac{y'}{r})^{n-1} y''=e^y\\
    y'(0)=0\\
    y''(0)=e^{\frac{y(0)}{n}}.
\end{cases}
\end{equation}
The metric so constructed is the first example of K\"ahler--Einstein metric (with Einstein constant $\lambda=-1$) which is not locally homogeneous (see also \cite{wolf} for an alternative proof of the fact that it is not locally homogeneous using Lie groups). 

Observe that we can see the tubular domain $M_n=\frac{1}{2}D\oplus i\R^n$ as the open submanifold of $\C^n$ given by
    \begin{equation*}
    \begin{split}
M_n&=\left\{z=(z_1,\dots,z_n)\in\C^n\;|\;\sum_{j=1}^n (z_j+\overline{z_j})^2<a^2\right\},
    \end{split}
    \end{equation*}
    where $a$ is the upper bound of the domain of regularity of $y(r)$.

\begin{remark}\rm
    By recursion from \eqref{cauchy}, one obtains that for all $j\in \mathds N$
    $$
    y^{(2j+1)}(0)=0,
    $$
    thus the power expansion around the origin of the function $y(r)$ is of the form
    \begin{equation}\label{espansione y}
        y(r)=y(0)+\sum_{j=1}^\infty b_{2j} r^{2j}
    \end{equation}
    where $b_{2j}=\frac{y^{(2j)}(0)}{(2j)!}$. In particular, the first three coefficients of the expansion are
     \begin{equation}\label{espansione y1}
    b_2=\frac{1}{2}e^{y(0)/2},\quad b_4=\frac{1}{32}e^{y(0)},\quad b_6=\frac{7}{2304}e^{3y(0)/2}.
   \end{equation}
\end{remark}

In order to prove Theorem \ref{main2}, we first have to show that $\mathcal{H}_\alpha\neq\{0\}$ for $(M_n,g_n)$, ensuring the existence of the $\varepsilon$-function. This is done in Lemma \ref{lemmaHneq0} below.

\begin{lemma}\label{lemmaHneq0}
In the notation above, $\mathcal H_\alpha\neq\{0\}$ for $(M_n, g_n)$.
\end{lemma}
\begin{proof}
First, notice that the volume form of $g_n$ is given by
    \[
    \frac{\omega_n^n}{n!}=\det(H)\frac{\omega_0^n}{n!}=e^{f}\frac{\omega_0^n}{n!}
    \]
    where $H=\left(\frac{\partial^2 f}{\partial x_i\partial x_j}\right)$ is the hessian of the K\"ahler potential $f$, $\omega_0$ is the standard euclidean form of $\C^n$, and the second equality follows since by construction we have $\det(f_{jk})=e^{f}$ (see \cite{calabi} p.19). Consider now the holomorphic function
\[
h(z)=\prod_{j=1}^n \frac{1}{z_j-2a}
\]
over $M_n=\frac{1}{2}D\oplus i\R^n$, where $a$ is the radius of the ball $D\subset\R^n$. For $j=1,\dots, n$, denote $x_j:=2{\rm Re}(z_j)$ (as before) and $u_j:=2{\rm Im}(z_j)$. Observe that in this notation,
$$
 \frac{\omega_0^n}{n!}=\left(\frac{i}{2}\right)^n \prod_{j=1}^n dz_j\wedge d\overline{z_j}=\frac{1}{2^{2n}}\prod_{j=1}^n dx_j\wedge du_j.
$$
We have
\[
|h(z)|^2=\prod_{j=1}^n \frac{1}{z_j-2a}\frac{1}{\overline{z_j}-2a}=\prod_{j=1}^n \frac{1}{\frac{x_j^2}{4}+\frac{u_j^2}{4}-2ax_j+4a^2}.
\]
Thus, %\int_{M_n} e^{-\alpha f}|h|^2 e^f \omega_0^n
\begin{equation*}
\begin{split}
\int_{M_n} e^{-\alpha f}|h|^2 \frac{\omega_n^n}{n!}&=\int_{M_n} e^{(1-\alpha)f}|h|^2\frac{\omega_0^n}{n!}
=\int_{\frac{1}{2}D}\int_{\R^n} e^{(1-\alpha)f}|h|^2\frac{\omega_0^n}{n!}\\
&=\frac{1}{2^{2n}}\int_{\frac{1}{2}D}\int_{\R^n} e^{(1-\alpha)f}\prod_{j=1}^n \frac{1}{\frac{x_j^2}{4}+\frac{y_j^2}{4}-2ax_j+4a^2}\prod_{j=1}^n dx_j\wedge dy_j\\
&=\frac{1}{2^{2n}}\int_{\frac{1}{2}D}e^{(1-\alpha)f}\prod_{j=1}^n dx_j\int_{-\infty}^{+\infty}\cdots\int_{-\infty}^{+\infty}\frac{1}{\frac{x_j^2}{4}+\frac{y_j^2}{4}-2ax_j+4a^2}\prod_{j=1}^n dy_j\\
&=\pi^n\int_{\frac{1}{2}D}\frac{e^{(1-\alpha)f}}{\prod_{j=1}^n|x_j-4a|}\prod_{j=1}^n dx_j.
\end{split}
\end{equation*}
Now, the function $\frac{1}{\prod_{j=1}^n|x_j-4a|}$ is bounded on $\frac{1}{2}D$ since $D$ has radius $a$  and similarly $e^{(1-\alpha)f}$ is bounded since $f$ is a smooth function diverging at $+\infty$ on the boundary of $D$, i.e. $e^{(1-\alpha)f}\to 0$ as $p\to\partial D$, for $\alpha>1$.
\end{proof}

By Theorem \ref{th: mamarinescu}, this proves the first part of Theorem \ref{main2}. To show the second part, namely that the $a_3$ coefficient of the TYCZ-expansion of the $\varepsilon$-function of $(M_n,g_n)$ is not zero, we start by describing more in details Calabi's metric for $n=2$. By definition, $g_2$ is given by: 
\[
(g_2)_{i\overline{j}}=
\frac{\partial^2 y(r)}{\partial x_i\partial{x_j}}=\frac{y'}r \delta_{ij}+\left(y''-\frac{y'}r\right)\frac{x_i x_j}{r^2},
\]
while its inverse is described by 
$$
g^{-1}_{i\overline{j}}=\frac{r}{y'}\delta_{ij}+\left(\frac{1}{y''}-\frac{r}{y'}\right)\frac{x_i x_j}{r^2},
$$
that is, the matrix representing $g_2$ is
\begin{equation}\label{G2}
G_2=\begin{bmatrix}\frac{y'}r+(y''-\frac{y'}r)\frac{x_1^2}{r^2}&(y''-\frac{y'}r)\frac{x_1x_2}{r^2}\\
(y''-\frac{y'}r)\frac{x_1x_2}{r^2}&\frac{y'}r+(y''-\frac{y'}r)\frac{x_2^2}{r^2}\end{bmatrix}
\end{equation}
while its inverse is
\begin{equation}\label{G2inv}
G_2^{-1}=
\begin{bmatrix}
    \frac{r}{y'}+(\frac{1}{y''}-\frac{r}{y'})\frac{x_1^2}{r^2}&(\frac{1}{y''}-\frac{r}{y'})\frac{x_1 x_2}{r^2}\\
    (\frac{1}{y''}-\frac{r}{y'})\frac{x_1 x_2}{r^2}&\frac{r}{y'}+(\frac{1}{y''}-\frac{r}{y'})\frac{x_2^2}{r^2}
\end{bmatrix}.
\end{equation}

Let us prove Theorem \ref{main2}.
\begin{proof}[Proof of Theorem \ref{main2}]
   Calabi's metric is K\"ahler--Einstein with Einstein constant $-1$, thus by Theorem \ref{main1} $a_3$ is different from zero for all $n\geq 3$. Set $n=2$.

   We divide the proof in 5 steps. In Step 1, Step 2 and Step 3 we prove that the vanishing of $a_3$ is equivalent to $y'\partial_r |R|^2=0$, that is, since $y'(r)> 0$ for any $r>0$, it is equivalent to $\partial_r |R|^2=0$ for $r\in (0,a)$. 
To conclude the proof we show in Step 4 and Step 5 that $|R|^2$ cannot be constant in $(0,a)$ since $\lim_{r\rightarrow 0}|R|^2\neq \lim_{r\rightarrow a}|R|^2$.
\begin{enumerate}
   \item[{\bf Step 1:}]  {\bf $ \Delta|R|^2=\frac{1}{y''}\partial^2_r|R|^2+\frac{1}{y'}\partial_r|R|^2$.}\\
    
 A straighforward computation using the formula for the Riemannian tensor for $g_n$ computed by Calabi (see \cite{calabi}, p. 23), we get:
\begin{equation}\label{riemanng2}
\frac12|R|^2=2-\frac{8re^y}{y'^3}+\frac{y'^{4}}{r
^{4}e^{2y}}+\frac 
{4}{ry'}+\frac {12{r}^{
2}e^{2y}}{ y'^{6}}-\frac{2
y'}{r^3e^y}-\frac{12e^y}{y'^{4}}+
\frac{6y'^{3}}{{r}^{5} e^{2y}}-\frac {12}{{r}^{4}e^y}+
\frac {12y'^
{2}}{{r}^{6} e^{2y}}.
\end{equation}
In particular, $|R|^2$ is a function of $r$, and by \eqref{laplace} we have:
        \begin{equation*}
\begin{split}
\Delta|R|^2=&\,g^{11}\frac{\partial^2}{\partial z_1\partial\overline{z}_1}|R|^2+g^{12}\frac{\partial^2}{\partial z_1\partial\overline{z}_2}|R|^2+g^{21}\frac{\partial^2}{\partial z_2\partial\overline{z}_1}|R|^2+g^{22}\frac{\partial^2}{\partial z_2\partial\overline{z}_2}|R|^2=\\
=&\,g^{11}\frac{\partial^2}{\partial x_1^2}|R|^2+2g^{12}\frac{\partial^2}{\partial x_1\partial x_2}|R|^2+g^{22}\frac{\partial^2}{\partial x_2^2}|R|^2\\
=&\,g^{11}\frac{\partial}{\partial x_1}\left(\partial_r|R|^2 \frac{x_1}{r}\right)+2g^{12}\frac{\partial}{\partial x_1}\left(\partial_r|R|^2\frac{x_2}r\right)+g^{22}\frac{\partial}{\partial x_2}\left(\partial_r|R|^2\frac{x_2}r\right)\\
=&\,g^{11}\left(\partial^2_r|R|^2 \frac{x_1^2}{r^2}+\frac1r\partial_r|R|^2\left(1-\frac{x_1^2}{r^2}\right)\right)+2g^{12}\left(\partial^2_r|R|^2 \frac{x_1x_2}{r^2}-\frac1r\partial_r|R|^2\frac{x_1x_2}{r^2}\right)+\\
&+g^{22}\left(\partial^2_r|R|^2 \frac{x_2^2}{r^2}+\frac1r\partial_r|R|^2\left(1-\frac{x_2^2}{r^2}\right)\right)\\
=&\,\partial^2_r|R|^2 \left(g^{11} \frac{x_1^2}{r^2}+2g^{12}\frac{x_1x_2}{r^2}+g^{22}\frac{x_2^2}{r^2}\right)+\frac1r\partial_r|R|^2\left(g^{11}+g^{22}-\left(g^{11} \frac{x_1^2}{r^2}+2g^{12}\frac{x_1x_2}{r^2}+g^{22}\frac{x_2^2}{r^2}\right)\right)\\
=&\frac{1}{y''}\partial^2_r|R|^2+\frac{1}{y'}\partial_r|R|^2.
\end{split}
\end{equation*}
where $g_{jk}$ and $g^{jk}$, $j$, $k=1,2$, represent respectively the entries of $G_2$ and $G_2^{-1}$, given in \eqref{G2} and \eqref{G2inv}.
%\end{proof}
\item[{\bf Step 2:}]  {\bf $\Delta|R|^2=0$ if and only if $y'\partial_r|R|^2$ is constant.}\\

By Step 1, $\Delta|R|^2=0$ if and only if:
$$
\frac{y''}{y'}=-\frac{\partial^2_r|R|^2}{{\partial_r|R|^2}}
$$
that is:
$$
\partial_r\left(\log(y')\right)=-\partial_r\left(\log(\partial_r|R|^2)\right),
$$
and thus for a constant $c$:
$$
\log(\partial_r|R|^2)=-\log(y')+c,
$$
that is:
$$
\partial_r|R|^2=\frac{c}{y'}.
$$

%\end{proof}
\item[{\bf Step 3:}] 
{\bf $\lim_{r\rightarrow a}y'\partial_r|R|^2=0$.}\\ 
   
From \eqref{riemanng2} it follows:
\begin{equation}
    \begin{split}
       \frac14\partial_r|R|^2=&-\frac{4re^y}{y'^2}+\frac{24r^2e^{2y}}{y'^5}-\frac{36r^3e^{3y}}{y'^8}-\frac{8y'^4}{r^5e^{2y}}-\frac{y'^5}{r^4e^{2y}}+\frac{3y'^2}{r^3e^y} -\frac{3}{r^2y'}+\\
       &+\frac{18y'}{r^4e^y}-\frac{27y'^3}{r^6e^{2y}}+\frac{36}{r^5e^y}-\frac{36y'^2}{r^7e^{2y}}+\frac{36re^{2y}}{y'^6}-\frac{12e^y}{y'^3}.
    \end{split}
\end{equation}
Let us write:
$$
\frac14y'\partial_r|R|^2=A+B+C,
$$
where:
\begin{equation*}
\begin{split}
A&:=-\frac{8y'^5}{r^5e^{2y}}+\frac{18y'^2}{r^4e^y}-\frac{27y'^4}{r^6 e^{2y}}+\frac{36y'}{r^5e^y}-\frac{36y'^3}{r^7e^{2y}},\\
B&:=-\frac{y'^6}{r^4e^{2y}}+\frac{3y'^3}{r^3e^y} -\frac{3}{r^2},\\
C&:=-\frac{4re^y}{y'}+\frac{24r^2e^{2y}}{y'^4}-\frac{36r^3e^{3y}}{y'^7}+\frac{36re^{2y}}{y'^5}-\frac{12e^y}{y'^2}.
\end{split}
\end{equation*}
In order to compute the limits for ${r\rightarrow a}$ of $A$, $B$ and $C$, we first observe that by construction (see \cite[p. 21]{calabi}):
$$
\lim_{r\rightarrow a}y=+\infty,\quad \lim_{r\rightarrow a}y'=+\infty,
$$
and by \eqref{cauchy}:
\begin{equation}\label{firstlimit}
 \lim_{r\rightarrow a}y''=\lim_{r\rightarrow a}\frac{e^{y}r}{y'}=\lim_{r\rightarrow a}\frac{(1+ry')y'}{r}=+\infty,
\end{equation}
\begin{equation}\label{secondlimit}
\lim_{r\rightarrow a}\frac{e^y}{y'^3}=\lim_{r\rightarrow a}\frac{e^yy'}{3y'^2y''}=\lim_{r\rightarrow a}\frac{1}{3r}=\frac{1}{3a},
\end{equation}
Further, $\lim_{r\rightarrow a}(y'^3-3re^y)$ is not finite,
in fact
\begin{equation*}
    \lim_{r\rightarrow a}(y'^3-3re^y)=\lim_{r\rightarrow a}y'^3\left(1-\frac{3re^y}{y'^3}\right)=\lim_{r\rightarrow a}\frac{1-\frac{3re^y}{y'^3}}{\frac{1}{y'^3}}=\frac{0}{0}
\end{equation*}
and applying de l'Hopital we get:
\begin{equation*}
    \lim_{r\rightarrow a}\frac{1-\frac{3re^y}{y'^3}}{\frac{1}{y'^3}}=\lim_{r\rightarrow a}\frac{\frac{-y'^3(3e^y+3re^y y')+9y'^2y''re^y}{y'^6}}{-\frac{3y'^2y''}{y'^6}}=\frac{y'^2}{r}+y'^3-3re^y.
\end{equation*}
Thus, if $\lim_{r\rightarrow a}\left(y'^3-3re^y\right)=c$, with 
$c\in\R$, then we get the contradiction:
\[
\lim_{r\rightarrow a}\left(y'^3-3re^y\right)=\lim_{r\rightarrow a}\left(\frac{y'^2}{r}+y'^3-3re^y\right)=+\infty.
\]
Therefore we can apply de l'Hopital to $\lim_{r\rightarrow a}\frac{y'^3-3re^{y}}{y'^2}$ and we get:
\begin{equation}\label{thirdlimit}
\lim_{r\rightarrow a}\frac{y'^3-3re^{y}}{y'^2}=\lim_{r\rightarrow a}\frac{3y'^2y''-3e^{y}-3re^yy'}{2y'y''}=\lim_{r\rightarrow a}\frac{3y're^y-3e^{y}-3re^yy'}{2re^y}=-\frac3{2a}.
\end{equation}
By \eqref{firstlimit} and \eqref{secondlimit} we have:
\begin{equation*}
    \lim_{r\rightarrow a}A=\lim_{r\rightarrow a}\left(-\frac{8y'^5}{r^5e^{2y}}+\frac{18y'^2}{r^4e^y}-\frac{27y'^4}{r^6e^{2y}}+\frac{36y'}{r^5e^y}-\frac{36y'^3}{r^7e^{2y}}\right)=0,
\end{equation*}
\begin{equation*}
    \lim_{r\rightarrow a}B=\lim_{r\rightarrow a}\left(-\frac{y'^6}{r^4e^{2y}}+\frac{3y'^3}{r^3e^y} -\frac{3}{r^2}\right)=-\frac3{a^2},
\end{equation*}
Finally by \eqref{secondlimit} and \eqref{thirdlimit} we have:
\begin{equation*}
\begin{split}
\lim_{r\rightarrow a}C&=\lim_{r\rightarrow a}\left(
-\frac{4re^y}{y'}+\frac{24r^2e^{2y}}{y'^4}-\frac{36r^3e^{3y}}{y'^7}+\frac{36re^{2y}}{y'^5}-\frac{12e^y}{y'^2}
\right)\\
&=\lim_{r\rightarrow a}\left[-y'^2\left(
4-\frac{24re^{y}}{y'^3}+\frac{36r^2e^{2y}}{y'^6}\right)+12\left(\frac{3e^{y}}{y'^3}-\frac{1}{r}\right)y'
\right]\frac{re^y}{y'^3}\\
&=\lim_{r\rightarrow a}\left[-\left(
\frac{y'^3-3re^{y}}{y'^2}\right)^2-\frac3r\frac{y'^3-3re^{y}}{y'^2}
\right]\frac{4re^y}{y'^3}\\
&=\left[-\frac9{4a^2}+\frac9{2a^2}
\right]\frac{4}{3}=\frac3{a^2},
\end{split}
\end{equation*}
and we are done.
\item[{\bf Step 4:}] {\bf $\lim_{r\rightarrow a}|R|^2=\frac43$.}\\
This follows directly from \eqref{riemanng2}.

\item[{\bf Step 5:}] {\bf $\lim_{r\rightarrow 0}|R|^2=\frac{3}{2}$.}\\

Using \eqref{riemanng2}, let us write:

\begin{equation*}%\label{Rlimit0}
    \begin{split}
\frac12|R|^2=&2-\frac{8re^y}{y'^3}+\frac{y'^{4}}{r
^{4}e^{2y}}+\frac 
{4}{ry'}+\frac {12{r}^{
2}e^{2y}}{ y'^{6}}-\frac{2
y'}{r^3e^y}-\frac{12e^y}{y'^{4}}+
\frac{6y'^{3}}{{r}^{5} e^{2y}}-\frac {12}{{r}^{4}e^y}+
\frac {12y'^
{2}}{{r}^{6} e^{2y}}\\
=&2+\frac{y'^{4}}{r
^{4}e^{2y}}+\frac{r^3}{y'^3}\left(-8e^{3y}+4\frac{y'^2e^{2y}}{r^2}-\frac{2
y'^4e^y}{r^4}+
\frac{6y'^{6}}{{r}^{6}}\right)\frac1{r^2}\frac{1}{ e^{2y}}+\\
&12\frac{r^6}{y'^6}\left(e^{4y}-\frac{e^{3y}y'^2}{r^2}-\frac{e^yy'^6}{r^6}+
\frac {y'^
{8}}{{r}^{8}}\right)\frac{1}{r^4}\frac1{e^{2y}}.
%=&2+1-9+18=12
  \end{split}
\end{equation*}
 Since 
 \[
 \lim_{r\rightarrow 0}\frac{y'^4}{r^4 e^{2y}}=1,
 \]
we are done by showing that:
 \begin{equation}\label{limite1}
        \lim_{r\rightarrow 0}\frac{r^3}{y'^3}\left(-8e^{3y}+4\frac{y'^2e^{2y}}{r^2}-\frac{2
y'^4e^y}{r^4}+
\frac{6y'^{6}}{{r}^{6}}\right)\frac1{r^2}\frac{1}{ e^{2y}}=-\frac92,
\end{equation} 
 \begin{equation}\label{limite2}
\lim_{r\rightarrow 0}\frac{r^6}{y'^6}\left(e^{4y}-\frac{e^{3y}y'^2}{r^2}-\frac{e^yy'^6}{r^6}+
\frac {y'^
{8}}{{r}^{8}}\right)\frac{1}{r^4}\frac1{e^{2y}}=\frac3{16}.
 \end{equation} 
 By \eqref{espansione y} and \eqref{espansione y1} we have:
\begin{equation}\label{limit0}
    \frac{y'(r)}{r}=e^{y(0)/2}+\sum_{j=1}^\infty c_{2j}r^{2j},%=e^{y(0)/2}+\frac1{8}e^{y(0)} r^2+\frac{7}{384}e^{3y(0)/2}r^4+o(r^6).
\end{equation}
with $c_2=\frac18e^{y(0)}$ and $c_4=\frac{7}{384}e^{3y(0)/2}$.
For shorten the notation let us write:
\begin{equation}\label{pqs}
\begin{split}
     P:=&\frac{y'}{r}=e^{y(0)/2}+\sum_{j=1}^\infty c_{2j}r^{2j},\\
        Q:=&\frac{P'}{r}=2c_2+\sum_{j=2}^\infty 2j c_{2j} r^{2j-2},\\
        S:=&\frac{Q'}{r}=8c_{4}+\sum_{j=3}^\infty 2j(2j-2) c_{2j} r^{2j-4}
\end{split}
    \end{equation}
To compute \eqref{limite1} let us start with the following:
\begin{equation*}
    \begin{split}
\lim_{r\rightarrow 0}&\left(-8e^{3y}+4\frac{y'^2e^{2y}}{r^2}-\frac{2
y'^4e^y}{r^4}+
\frac{6y'^{6}}{{r}^{6}}\right)\frac1{r^2}=\lim_{r\rightarrow 0}\frac{-8e^{3y}+4e^{2y}P^2-2e^y
P^4+
6P^{6}}{r^2}\\
=&\lim_{r\rightarrow 0}\left(\frac{-24e^{3y}y'+8e^{2y}y'P^2+8e^{2y}PP'-8e^y
P^3P'-2e^yy'P^4+
36P^{5}P'}{2r}\right)\\
=&\frac12\lim_{r\rightarrow 0}\left(-24e^{3y}P+8e^{2y}P^3+8e^{2y}PQ-8e^y
P^3Q-2e^yP^5+
36P^{5}Q\right)\\
=&-\frac{9}2e^{7y(0)/2},
 \end{split}
\end{equation*}
where we applied de l'Hopital and used \eqref{pqs}.
Plugging the result into \eqref{limite1} we get:
\[
\lim_{r\rightarrow 0}\frac{r^3}{y'^3}\left(-8e^{3y}+4\frac{y'^2e^{2y}}{r^2}-\frac{2
y'^4e^y}{r^4}+
\frac{6y'^{6}}{{r}^{6}}\right)\frac1{r^2}\frac{1}{ e^{2y}}=-\frac{9}{2}.
\]
To compute \eqref{limite2}, we need to apply de l'Hopital twice to the following limit:
\begin{equation*}
    \begin{split}
    \lim_{r\rightarrow 0}&\left(e^{4y}-\frac{e^{3y}y'^2}{r^2}-\frac{e^yy'^6}{r^6}+
\frac {y'^
{8}}{{r}^{8}}\right)\frac1{r^4}=\lim_{r\rightarrow 0}\frac{e^{4y}-e^{3y}P^2-e^yP^6+
P^
{8}}{r^4}\\
=&\, \lim_{r\rightarrow 0}\left(\frac{4e^{4y}y'-3e^{3y}y'P^2-2e^{3y}PP'-e^yy'P^6-6e^yP^5P'+8P^7P'}{4r^3}\right)\\
=&\, \frac14\lim_{r\rightarrow 0}\frac1{r^2}\left(4e^{4y}P-3e^{3y}P^3-2e^{3y}PQ-e^yP^7-6e^yP^5Q+8P^7Q\right)\\
=&\, \frac18\lim_{r\rightarrow 0}\frac1{r}\left(16y'e^{4y}P+4e^{4y}P'-9y'e^{3y}P^3-9e^{3y}P^2P'-6y'e^{3y}PQ-2e^{3y}P'Q-2e^{3y}PQ'-y'e^yP^7+\right.\\
&\qquad\quad\left.-7e^yP^6P'-6y'e^yP^5Q-30e^yP^4P'Q-6e^yP^5Q'+56P^6P'Q+8P^7Q'\right)\\
=&\, \frac18\lim_{r\rightarrow 0}\left(16e^{4y}P^2+4e^{4y}Q-9e^{3y}P^4-9e^{3y}P^2Q-6e^{3y}P^2Q-2e^{3y}Q^2-2e^{3y}PS-e^yP^8+\right.\\
&\qquad\quad\left.-7e^yP^6Q-6e^yP^6Q-30e^yP^4Q^2-6e^yP^5S+56P^6Q^2+8P^7S\right)\\
=&\,\frac{3}{16}e^{5y(0)},
    \end{split}
\end{equation*}
thus
\[
\lim_{r\rightarrow 0}\frac{r^6}{y'^6}\left(e^{4y}-\frac{e^{3y}y'^2}{r^2}-\frac{e^yy'^6}{r^6}+
\frac {y'^
{8}}{{r}^{8}}\right)\frac{1}{r^4}\frac1{e^{2y}}=\frac3{16},
\]
concluding the proof.
\end{enumerate}
\end{proof}

\end{document}